\documentclass[11pt]{amsart}

\usepackage{amssymb}
\usepackage{amsmath}
\usepackage{enumerate}
\usepackage[latin1]{inputenc}
\usepackage[dvips]{graphicx}

\def\s{\mathbb{S}}
\def\h{\mathbb{H}}
\def\r{\mathbb{R}}
\def\z{\mathbb{Z}}

\def\p{\mathbb{P}}

\newtheorem{theorem}{Theorem}

\newtheorem{corollary}{Corollary}

\theoremstyle{definition}

\theoremstyle{remark}
  \newtheorem{remark}{Remark}

\numberwithin{equation}{section}

\begin{document}

\title[Gauss curvature of surfaces in homogeneous 3-manifolds]{On the Gauss curvature of compact surfaces in homogeneous 3-manifolds}

\author{Francisco Torralbo}
\address{Departamento de Geometr\'{\i}a  y Topolog\'{\i}a \\
Universidad de Granada \\
18071 Granada, SPAIN} \email{ftorralbo@ugr.es}

\author{Francisco Urbano}
\address{Departamento de Geometr\'{\i}a  y Topolog\'{\i}a \\
Universidad de Granada \\
18071 Granada, SPAIN} \email{furbano@ugr.es}

\thanks{Research partially supported by a MCyT-Feder research project MTM2007-61775 and a Junta Andalucía  Grant P06-FQM-01642.}

\subjclass[2000]{Primary 53C42; Secondary 53C30}

\keywords{Surfaces, Gauss curvature, homogenous 3-manifolds}

%\date{}

\begin{abstract}
Compact flat surfaces of homogeneous Riemannian $3$-manifolds with isometry group of dimension $4$ are classified. Non-existence results for compact constant Gauss curvature surfaces in these $3$-manifolds are established.
\end{abstract}

\maketitle

\section{Introduction}
The study of complete constant Gauss curvature surfaces in simply-connected 3-manifolds of constant curvature is a classical topic in classical differential geometry. Perhaps one of the nicest reviews about it can be found in Spivak's book~\cite{S}.

During the last lustrum, the surfaces of homogeneous Riemannian 3-manifolds are been deeply studied and particulary those with constant mean curvature. The started point of this study was the work of Abresch and Rosenberg~\cite{AR} where they found out a holomorphic $2$-differential in any constant mean curvature surface of a homogeneous Riemannian $3$-manifold with isometry group of dimension $4$. The Berger $3$-spheres, the Heisenberg group, the Lie group $Sl(2,\r)$ and the Riemannian product $\s^2\times\r$ or $\h^2\times\r$, where $\s^2$ and $\h^2$ are the $2$-dimensional sphere and hyperbolic plane with their standard metrics, are the most significant examples of such homogeneous $3$-manifolds.

In the important and complete work~\cite{AEG}, Aledo, Espinar and Galvez classified the complete constant Gauss curvature surfaces in $\s^2\times\r$ and $\h^2\times\r$ for any value of the Gauss curvature except for two exceptional values, where the classification is still open.

In this paper we start the study of the Gauss curvature of {\it compact} surfaces of homogeneous Riemannian $3$-manifolds with isometry group of dimension $4$. We get a nice integral formula (formula~\eqref{eq:integral-formula}) which allows to classify the flat compact surfaces (Theorem~\ref{thm:flat-compact-surfaces-are-Hopf-tori}). This formula is new in the sense that there is not a similar one for surfaces in 3-dimensional space forms and it shows an important difference between flat compact surfaces in these $3$-manifolds and in $3$-dimensional space forms. For instance, in the Berger spheres the only flat compact surfaces are the Hopf tori, i.e., inverse image of closed curves of $\s^2$ by the Hopf fibration $\Pi:\s^3\rightarrow\s^2$, whereas in the standard $3$-sphere, besides them, there are other flat tori.

Using arguments of topological nature, in Theorem~\ref{thm:compact-surfaces} and Corollary~\ref{cor:compact-surface-constant-curvature}, we study the behavior of the Gauss curvature of this kind of surfaces and we prove non-existence results of compact surfaces with constant Gauss curvature in these homogeneous Riemannian $3$-manifolds.

\section{Preliminaries}
    In this section we follow the notation given in~\cite{D}. Let $N$ be a homogeneous Riemannian 3-manifold whose isometry group have dimension 4. Then there exists a Riemannian submersion $\Pi:N\rightarrow M^2(\kappa)$, where $M^2(\kappa)$ is a 2-dimensional simply connected space form of constant curvature $\kappa$, with totally geodesic fibers and there exists a unit Killing field $\xi$ on $N$ which is vertical with respect to $\Pi$. We will assume that $N$ is oriented, and then we can define a vectorial product $\wedge$, such that if $\{e_1,e_2\}$ are linear independent vectors at a point $p$, then $\{e_1,e_2,e_1\wedge e_2\}$ is the orientation at $p$. If $\bar{\nabla}$ denotes the Riemannian connection on $N$, the properties of $\xi$ imply that (see \cite{D}) for any vector field $V$
\begin{equation}\label{eq:killing-property}
\bar{\nabla}_V\xi=\tau(V\wedge\xi),
\end{equation}
where the constant $\tau$ is the bundle curvature. As the isometry group of $N$ has dimension 4, $\kappa-4\tau^2\not=0$. The case $\kappa-4\tau^2=0$ corresponds to $\s^3$ with its standard metric if $\tau\not=0$ and to the Euclidean space $\r^3$ if $\tau=0$, which have isometry groups of dimension $6$.

 When $\tau=0$, $N$ is the product $M^2(\kappa)\times\r$ or $M^2(\kappa)\times\s^1$, where $\s^1$ is a circle of certain radius.

 When $\tau\not=0$ and $N$ is simply-connected, $N$ is one of the Berger spheres for $\kappa>0$, the Heisenberg group $Nil_3$ for $\kappa=0$ and the universal cover $\widetilde{Sl(2,\r)}$ of the Lie group $Sl(2,\r)$ for $\kappa<0$. The parameters $\kappa,\tau$ give a $2$-parameter family of Berger metrics on $\s^3$ and of homogeneous metrics on $\widetilde{Sl(2,\r)}$. The parameter $\tau$ gives a $1$-parameter family of homogeneous metrics on $Nil_3$.

The real projective space $\r\p^3$ or more generally the Lens spaces $L_n=\s^3/\z_n,\,n\geq 3$ with the induced Berger metrics are also homogeneous Riemannian $3$-manifolds with isometry group of dimension $4$. Also, $PSl(2,\r)=Sl(2,\r)/\{\pm I\}$ or more generally $Sl(2,\r)/\z_n,\,n\geq 3$ are homogeneous Riemannian $3$-manifolds with isometry group of dimension $4$ too. In these cases, the corresponding Riemannian submersions on $M^2(\kappa)$ is by circles.

{\it Along the paper $E(\kappa,\tau)$ will denote an oriented homogeneous Riemannian $3$-manifold with isometry group of dimension $4$, where $\kappa$ is the curvature of the basis, $\tau$ the bundle curvature and then $\kappa-4\tau^2\not=0$.}

Now, let $\Phi:\Sigma\rightarrow E(\kappa,\tau)$ be an immersion of an orientable surface $\Sigma$, and $N$ a unit normal field on $\Sigma$. We define a function $C$ on $\Sigma$ by
\[
C=\langle\xi,N\rangle,
\]
where $\langle,\rangle$ stands for the metric on $E(\kappa,\tau)$ and $\Sigma$.
It is clear that $-1\leq C\leq 1$ and that $\{p\in\Sigma\,|\,C^2(p)=1\}=\{p\in\Sigma\,|\,\xi(p)=\pm N(p)\}$ has empty interior, because the distribution $\langle\xi\rangle^{\bot}$
on $E(\kappa,\tau)$ is not integrable. If $A$ is the shape operator of $\Phi$ associated to $N$ and $K$ is the Gauss curvature of $\Sigma$, then the Gauss equation of $\Phi$ is given by (see~\cite{D})
\begin{equation}\label{eq:gauss-equation}
K=\det\,A+\tau^2+(\kappa-4\tau^2)C^2=2H^2-\frac{|\sigma|^2}{2}+\tau^2+(\kappa-4\tau^2)C^2,
\end{equation}
where $H$ is the mean curvature associated to $N$ and $\sigma$ is the second fundamental form of $\Phi$.

As there exists a Riemannian submersion $\Pi:E(\kappa,\tau)\rightarrow M^2(\kappa)$, we can construct flat surfaces in $E(\kappa,\tau)$ in the following way. Given any regular curve $\gamma$ in $M^2(\kappa)$, $\Pi^{-1}(\gamma)$ is a surface in $E(\kappa,\tau)$ which has to $\xi$ as a tangent vector field, i.e., $C=0$. From~\eqref{eq:killing-property} it follows that $\xi$ is a parallel vector field on $\Sigma$ and hence $\Pi^{-1}(\gamma)$ is a flat surface. We will call to $\Pi^{-1}(\gamma)$ a {\it Hopf surface} of $E(\kappa,\tau)$. If $\gamma$ is a closed curve, $\Pi^{-1}(\gamma)$ is a flat cylinder and moreover if $\Pi$ is a circle Riemannian submersion, $\Pi^{-1}(\gamma)$ is a flat torus.
\section{Statement of results}
We start classifying the compact flat surfaces of $E(\kappa,\tau)$.
\begin{theorem}\label{thm:flat-compact-surfaces-are-Hopf-tori}
The only flat compact surfaces in $E(\kappa,\tau)$ are the Hopf tori. In particular
\begin{itemize}
\item In $M^2(\kappa)\times\r$, $\kappa\not=0$, the Heisenberg group $Nil_3$ and $\widetilde{Sl(2,\r)}$ there are no flat compact surfaces.
\item In $M^2(\kappa)\times\s^1$, $\kappa\not=0$ the only flat compact surfaces are the products $\gamma\times\s^1$, where $\gamma$ is a closed curve in $M^2(\kappa)$.
\item In the Berger spheres, $Sl(2,\r)$ and their quotiens $L_n$ and $Sl(2,\r)/\z_n$ the flat compact surfaces are the Hopf tori.
\end{itemize}
\end{theorem}
\begin{remark}
 The result in the Berger spheres contrasts with the case of the round sphere where, besides the Hopf tori, there are other flat tori (see~\cite{S} Chapter IV).
\end{remark}
\begin{proof}
First we are going to prove an integral formula for any compact surface of $E(\kappa,\tau)$. Let $\Phi:\Sigma\rightarrow E(\kappa,\tau)$ be an immersion of an oriented compact surface $\Sigma$ and $X$ the tangential component of the Killing field $\xi$, i.e. $$X=\xi-CN.$$ Then from~\eqref{eq:killing-property} and for any tangent vector $v$ to $\Sigma$ we obtain that
\begin{equation}
\label{eq:covariant-derivative-X}
\nabla_vX=\tau C(v\wedge N)+CAv,
\end{equation}
where $\nabla$ is the Riemannian connection on $\Sigma$. From~\eqref{eq:covariant-derivative-X} we obtain that
\begin{equation}\label{eq:divergence-X}
\text{div}\,X=2CH,\quad d(X^{\beta})(v,w)=2\tau C\langle v\wedge N,w\rangle,
\end{equation}
where $\text{div}$ is the divergence operator and $X^{\beta}$ is the 1-form associated to $X$, i.e., $X^{\beta}(v)=\langle X,v\rangle$. Now the classical Bochner formula says that
\[
\text{div}\,(\nabla_XX)=K|X|^2+\langle \nabla(\text{div}\,X),X\rangle+|\nabla X|^2+\sum_{i=1}^2d(X^{\beta})(\nabla_{e_i}X,e_i),
\]
where $\{e_1,e_2\}$ is an orthonormal basis in $T\Sigma$. Using~\eqref{eq:covariant-derivative-X} and~\eqref{eq:divergence-X} this equation becomes in
\[
\text{div}\,(\nabla_XX)=K(1-C^2)+\langle \nabla(\text{div}\,X),X\rangle+C^2|\sigma|^2-2\tau^2C^2.
\]
From~\eqref{eq:divergence-X} we get that
\[
\text{div}\,(2CHX)=\text{div}((\text{div}\,X)X)=\langle \nabla(\text{div}\,X),X\rangle+(\text{div}\,X)^2.
\]
The last two formulae joint with~\eqref{eq:divergence-X} again gives
\[
\text{div}\,(\nabla_XX-2CHX)=K(1-C^2)-4H^2C^2+C^2|\sigma|^2-2\tau^2C^2.
\]
Using the Gauss equation~\eqref{eq:gauss-equation} in the above formula, we obtain
\[
\text{div}\,(\nabla_XX-2CHX)=K(1-3C^2)+2(\kappa-4\tau^2)C^4.
\]
Finally, using the divergence theorem, the last formula becomes in the following integral formula for any orientable compact surface of $E(\kappa,\tau)$
\begin{equation}
\int_{\Sigma}K(1-3C^2)\,dv+2(\kappa-4\tau^2)\int_{\Sigma}C^4dv=0.
\label{eq:integral-formula}
\end{equation}
Although the function $C$ is only well defined for orientable surfaces, it is clear that $C^2$ is well defined always, even for non orientable surfaces. So  formula~\eqref{eq:integral-formula} is true for any compact surface of $E(\kappa,\tau)$.

Now we prove the Theorem. If $K=0$, \eqref{eq:integral-formula} becomes in
\[
(\kappa-4\tau^2)\int_{\Sigma}C^4dv=0,
\]
which implies that $C=0$ on $\Sigma$. This means that $X=\xi$, and so the vertical Killing field of the fibration $\Pi:E(\kappa,\tau)\rightarrow M^2(\kappa)$ is tangent to $\Sigma$. So $\Pi(\Sigma)$ is a closed curve $\gamma$ in $M^2(\kappa)$ and then $\Pi^{-1}(\gamma)\supset \Sigma$, which proves that $\Sigma$ is the Hopf torus $\Pi^{-1}(\gamma)$.
\end{proof}
\begin{remark}
The above formula~\eqref{eq:integral-formula} is also true for compact surfaces in other $3$-manifolds. In fact, suppose that $\Sigma$ is a compact surface of the Riemannian product $M^2\times\r$ or $M^2\times\s^1$, where $M^2$ is an oriented surface. In this case $\Pi:M^2\times\r\rightarrow M^2$ is the trivial submersion and so $\xi=(0,1)$ is a parallel vector field. Following the proof of~\eqref{eq:integral-formula}, if $X$ is the tangential component of $\xi$, then $\text{div}\,X=2CH$, $dX^{\beta}=0$ and $\text{div}\,(\nabla_XX-2CHX)=K(1-3C^2)+\bar{K}C^4$, where $\bar{K}$ is the Gauss curvature of the surface $M^2$. Now the divergence theorem says that
\[
\int_{\Sigma}(K(1-3C^2)+\bar{K}C^4)\,dv=0.
\]
Hence following the proof of Theorem~\ref{thm:flat-compact-surfaces-are-Hopf-tori} we can prove:
\begin{quote}
\emph{Let $M^2$ be an oriented surface with Gauss curvature $\bar{K}>0$ or $\bar{K}<0$. Then}
\begin{enumerate}
  \item \emph{There exist no compact flat surfaces in $M^2\times\r$.}
  \item \emph{The only compact flat surfaces in $M^2\times\s^1$ are the products $\gamma\times\s^1$ where $\gamma$ is a closed curve in $M^2$.}
\end{enumerate}
\end{quote}
\end{remark}

\begin{corollary}\label{cor:compact-C-constant->Hopf-tori}
The only compact surfaces in $E(\kappa,\tau)$ with $C$ constant are the Hopf tori, for which $C=0$.
\end{corollary}

\begin{remark}
This result is not true for non-compact surfaces. In fact, in~\cite{L} Leite constructed for each $K_0\in]-1,0[$ a constant mean curvature embedding of the hyperbolic plane with its standard metric of negative constant curvature $K_0$ in $M^2(-1)\times\r$, with constant function $C=-K_0$.
\end{remark}

\begin{proof}
It is clear that $C^2<1$. Then the tangential component $X$ of the Killing field is a vector field on $\Sigma$ without zeroes. So the Euler-Poincare characteristic $\chi(\Sigma)=0$ and then $\int_{\Sigma}K\,dv=0$. Now, as $C$ is constant, the formula~\eqref{eq:integral-formula} becomes in $(\kappa-4\tau^2)C^4\,\text{Area}\,(\Sigma)=0$, and hence $C=0$. We finish as in the proof of Theorem~\ref{thm:flat-compact-surfaces-are-Hopf-tori}.
\end{proof}

\begin{theorem}\label{thm:compact-surfaces}~
\begin{enumerate}[a)]
  \item \label{thm:compact-surfaces-item:non-exist}
  There exist no compact surfaces in $E(\kappa,\tau)$ with Gauss curvature $K<\min\,\{0,\kappa-4\tau^2\}$.

  \item \label{thm:compact-surfaces-item:K-bound}
  Let $\Phi:\Sigma\rightarrow E(\kappa,\tau)$ be an immersion of a compact surface $\Sigma$ with Gauss curvature $K$.

    \begin{enumerate}[(1)]
    \item If $\kappa-4\tau^2>0$ and $0\leq K<\max\{\kappa-4\tau^2,\tau^2\}$ then $K=0$ and $\Sigma$ is a Hopf torus.
    \item If $\kappa-4\tau^2< K\leq 0$ then $K=0$ and $\Sigma$ is a Hopf torus.
    \item If $\kappa-4\tau^2<0\leq K<\kappa-3\tau^2$ ($\kappa$ must be positive) then $K=0$ and $\Sigma$ is a Hopf torus.
    \end{enumerate}
\end{enumerate}
\end{theorem}

\begin{proof}
First we prove a).
Considering the double oriented cover, if necessary, we can assume that the surface $\Sigma$ is orientable. Suppose that $\Phi:\Sigma\rightarrow E(\kappa,\tau)$ is an immersion of an orientable compact surface with Gauss curvature $K<\min\,\{0,\kappa-4\tau^2\}$. Let $X$ be the tangential component of the Killing field $\xi$.
From~\eqref{eq:killing-property} and for any tangent vector $v$ to $\Sigma$, we get
\[
\tau(v\wedge X)=\sigma(v,X)+\langle\nabla C,v\rangle N,
\]
and so
\begin{equation}
\nabla C=\tau(X\wedge N)-AX.
\label{eq:gradient}
\end{equation}
It is clear that the points $p$ with $C^2(p)=1$, that is $X_p=0$, are critical points of $C$. Suppose that $p$ is a critical point of $C$ with $C^2(p)<1$. As in this case $X_p\not=0$, equation~\eqref{eq:gradient} says that the $AX_p=\tau(X_p\wedge N_p)$  and hence $\det\,A(p)=-\tau^2$. The Gauss equation~\eqref{eq:gauss-equation} implies that $K(p)=(\kappa-4\tau^2)C^2(p)$. As $0\leq C^2(p)\leq 1$ the assumptions on the Gauss curvature mean that such critical point $p$ of $C$ cannot exist. Hence the only critical points of $C$ are those with $C^2(p)=1$.

From~\eqref{eq:gradient} and~\eqref{eq:covariant-derivative-X} it is not difficult to check that the Hessian of $C$ in a critical point $p$, where $C^2(p)=1$ and $X_p=0$, is given by
\begin{eqnarray*}
(\nabla^2C)(v,w)=-C(p)\langle A^2v,w\rangle-\tau^2C(p)\langle v\wedge N_p,w\wedge N_p\rangle\\-\tau C(p)\langle Av\wedge w+Aw\wedge v,N_p\rangle,
\end{eqnarray*}
where $v,w$ are tangent vectors to $\Sigma$ at $p$. Hence, if $\{e_1,e_2\}$ is an orthonormal reference in $T_p\Sigma$ with $Ae_i=\lambda_ie_i,\,i=1,2$, the determinant of the Hessian of $C$ at $p$ is given by
\[
\det(\nabla^2C)(p)=(\lambda_1\lambda_2+\tau^2)^2=(K(p)-(\kappa-4\tau^2))^2>0.
\]
This means that $C$ is a Morse function on $\Sigma$ with only absolute maximum and minimum as critical points. The elementary Morse theory says that $\Sigma $ must be a sphere. As $K$ is negative this contradicts the Gauss-Bonnet theorem. This proves a).

Proof of b). As in the case a), we can assume that our surface is orientable.

 First we suppose that either $\kappa-4\tau^2<K\leq 0$ or $0\leq K<\kappa-4\tau^2$.

 Now let $p$ be a zero of the vector field $X$ defined by $X=\xi-CN$. Then following~\cite{M}, Chapter 6, if $(\nabla X)_p:T_p\Sigma\rightarrow T_p\Sigma$ denotes the derivative of $X$ at the point $p$, then $p$ is a non-degenerate zero of $X$ if $\det\,(\nabla X)_p\not=0$.

From~\eqref{eq:covariant-derivative-X}, if $\{e_1,e_2\}$ is an orthonormal reference of $T_p\Sigma$, then
\[
\langle\nabla_{e_i}X,e_j\rangle=C(p)\langle Ae_i,e_j\rangle-\tau C(p)\langle e_i\wedge e_j,N(p)\rangle,
\]
and so taking into account the Gauss equation~\eqref{eq:gauss-equation},  $\det\,(\nabla X)_p=\det\,A(p)+\tau^2=K(p)-(\kappa-4\tau^2)$. Using the assumptions on the curvature $K$, this proves that such zero $p$ of $X$ is non-degenerate. So $X$ has a finite number of zeroes.

Now as the index of $X$ at a non-degenerate zero $p$ is $1$ (respectively $-1$) if $\det\,(\nabla X)_p>0$ (respectively  $\det\,(\nabla X)_p<0$), we have that if $K>\kappa-4\tau^2$ all the zeroes of $X$ (if they exist) have index $1$, and then the Poincare-Hopf theorem says that the Euler-Poincare characteristic $\chi(\Sigma)$ is $2$ if $X$ has zeroes or $\chi(\Sigma)=0$ if $X$ has no zeroes. As in this case $K\leq 0$, the Gauss-Bonnet theorem says that $K=0$ and so Theorem~\ref{thm:flat-compact-surfaces-are-Hopf-tori} proves that $\Sigma$ is a Hopf torus. This proves (2).

  If $K<\kappa-4\tau^2$, then all the zeroes of $X$ (if they exist) have index $-1$, and then the Poincare-Hopf theorem says that $\chi(\Sigma)$ is negative if $X$ has zeroes or $\chi(\Sigma)=0$ if $X$ has no zeroes. As in this case $K\geq 0$, the Gauss-Bonnet theorem says again that $K=0$ and so $\Sigma$ is a Hopf torus.

Suppose now that $\kappa-4\tau^2>0$ and $0\leq K<\tau^2$.

The Gauss equation of $\Phi$ says that
\[
\tau^2>K=\det\,A+\tau^2+(\kappa-4\tau^2)C^2\geq\det\,A+\tau^2.
\]
Hence $\det\,A<0$ on $\Sigma$ and so if $\lambda_1\geq\lambda_2$ are the principal curvatures of $\Sigma$ associated to $N$, then $\lambda_2<0<\lambda_1$ and hence the tangent bundle $T\Sigma$ is the direct sum of two rank $1$ vector sub-bundles over $\Sigma$, $T\Sigma={D}_1\oplus {D}_2$ define by $D_i(p)=\{v\in T_p\Sigma\,|\,Av=\lambda_i(p)v\}, \,i=1,2$. As a sphere does not admit neither vector fields without zeroes nor $2$-folds covers, on a sphere is imposible such decomposition of its tangent bundle.  Hence, $\chi(\Sigma)\leq 0$. As $K\geq 0$, the Gauss-Bonnet theorem says that $\Sigma$ is flat and Theorem~\ref{thm:flat-compact-surfaces-are-Hopf-tori} that $\Sigma$ is a Hopf torus. This finishes the proof of (1).

 Finally suppose that $\kappa-4\tau^2<0\leq K<\kappa-3\tau^2$. In this case, as $C^2\leq 1$, we have
\[
\kappa-3\tau^2>K=\det\,A+\tau^2+(\kappa-4\tau^2)C^2\geq\det\,A+\kappa-3\tau^2.
\]
Hence $\det\,A<0$ on $\Sigma$ and the above reasoning proves again that $K=0$ and $\Sigma$ is a Hopf torus. This proves (3).
\end{proof}

\begin{corollary}\label{cor:compact-surface-constant-curvature}
Let $\Phi:\Sigma\rightarrow E(\kappa,\tau)$ be an immersion of a compact surface $\Sigma$ with {\bf constant} Gauss curvature $K$.
\begin{enumerate}
  \item If $\kappa-4\tau^2>0$ and $K\in\,]-\infty,\max\,\{\kappa-4\tau^2,\tau^2\}[$, then $K=0$ and $\Sigma$ is a Hopf torus.
  \item If $\kappa - 4\tau^2 < 0 <\kappa-3\tau^2$ and $K\in\,]-\infty,\kappa-3\tau^2[$, then either $K=0$ and $\Sigma$ is a Hopf torus or $K=\kappa-4\tau^2$.
  \item If $\kappa-4\tau^2< 0$ and $K\in\,]-\infty,0]$, then either $K=0$ and $\Sigma$ is a Hopf torus or $K=\kappa-4\tau^2$.
\end{enumerate}
\end{corollary}

\begin{remark}
When $\tau=0$, this Corollary was proved in~\cite{AEG}. In fact they classified the complete constant Gauss curvature surfaces of $M^2(\kappa)\times\r$, $\kappa\not=0$, except for $K=\kappa$.
\end{remark}

\end{document}